\documentclass[12pt]{article}
\usepackage[top=2.54cm, bottom=2.54cm, left=2.5cm, right=2.5cm]{geometry}
\usepackage[tbtags]{amsmath}
\usepackage{amssymb}
\usepackage{amsthm}
\usepackage{fancyhdr}
\usepackage{latexsym}
\usepackage{mathrsfs}
\usepackage{xcolor}
\usepackage{wasysym}
\usepackage{fancyhdr}
\usepackage{float}
\usepackage{graphicx}
\usepackage[numbers,sort&compress]{natbib}
\usepackage{pict2e,color}
\usepackage{tikz}
\usepackage{tkz-graph}
\usetikzlibrary{arrows,shapes,chains}
\usepackage{enumerate}
\usepackage {verbatim}
\usepackage{tikz-cd}
\usepackage{rotating}
\usepackage{subcaption}
\allowdisplaybreaks[4]

\usepackage{url}

\newtheorem{dfn}{Definition}
\newtheorem{theorem}[dfn]{Theorem}
\newtheorem{lemma}[dfn]{Lemma}

\newtheorem{conjecture}[dfn]{Conjecture}
\newtheorem{problem}[dfn]{Problem}

\newtheorem{definition}{Definition}

\def\l{\left(}
\def\r{\right)}

\let\svthefootnote\thefootnote
\newcommand\blankfootnote[1]{%
	\let\thefootnote\relax\footnotetext{#1}%
	\let\thefootnote\svthefootnote%
}

\begin{document}

\title{A problem of Erd\H{o}s and Hajnal on paths with equal-degree endpoints}

\author{
Kaizhe Chen\thanks{School of the Gifted Young, University of Science and Technology of China, Hefei, Anhui 230026, China. Email: ckz22000259@mail.ustc.edu.cn.}
~~~~
Jie Ma\thanks{School of Mathematical Sciences, University of Science and Technology of China, Hefei, Anhui 230026, China, and Yau Mathematical Sciences Center, Tsinghua University, Beijing 100084, China.
Research supported by National Key Research and Development Program of China 2023YFA1010201 and National Natural Science Foundation of China grant 12125106. Email: jiema@ustc.edu.cn.}
}

\date{}


\maketitle

\begin{abstract}
We address a problem posed by Erd\H{o}s and Hajnal in 1991, proving that for all $n \geq 600$,  
every $(2n+1)$-vertex graph with at least $n^2 + n + 1$ edges contains two vertices of equal degree connected by a path of length three.
The complete bipartite graph $K_{n,n+1}$ demonstrates that this edge bound is sharp.
We further establish an analogous result for graphs with even order and investigate several related extremal problems.
\end{abstract}

\section{Introduction}
A foundational result in graph theory asserts that every finite graph contains at least two vertices with equal degree. 
While infinitely many graphs can be constructed with exactly one such pair, 
this raises a natural question: what additional properties must these vertices satisfy under certain conditions?
In 1991, Erd\H{o}s and Hajnal \cite{E91} proposed a related extremal problem, 
seeking the minimum edge density that guarantees two equal-degree vertices connected by a path of length three.
They noted that complete bipartite graphs with two parts of different sizes do not satisfy this property and asked whether any graph exceeding the maximum edge count must contain such vertices.
Their precise formulation is as follows.

\begin{problem}[Erd\H{o}s-Hajnal, \cite{E91}]\label{main problem}
Is it true that every $(2n+1)$-vertex graph with $n^2+n+1$ edges contains two vertices of the same degree that are joined by a path of length three?
\end{problem}

This problem is also listed as Problem~\#816 in Thomas Bloom's collection of Erd\H{o}s problems \cite{E816}. 
The bound on the number of edges would be sharp if true, as demonstrated by the complete bipartite graph $K_{n,n+1}$.

The main contribution of this paper is the following theorem, 
which provides a strengthened affirmative solution to the Erd\H{o}s-Hajnal problem stated above.

\begin{theorem}\label{main-theorem}
Let $n\geq 600$. The unique $(2n+1)$-vertex graph with at least $n^2+n$ edges, that does not contain two vertices of the same degree joined by a path of length three, is the complete bipartite graph $K_{n,n+1}$.
\end{theorem}

This result strengthens Problem~\ref{main problem} in two respects. 
First, it shows that the complete bipartite graph $K_{n,n+1}$ is the unique extremal graph. 
Second, we note that the property of having two vertices of equal degree connected by a path of length three is not a monotone graph property, 
so proving the statement for graphs with at least $n^2+n+1$ edges is indeed stronger than for those with exactly $n^2+n+1$ edges.

Our method also extends to graphs with an even number of vertices, as follows. 

\begin{theorem}\label{main-theorem2}
There exists an integer $n_0>0$ such that the following holds for all $n\geq n_0$. 
The unique $2n$-vertex graph with at least $n^2-1$ edges, 
that does not contain two vertices of the same degree joined by a path of length three, is the complete bipartite graph $K_{n-1,n+1}$.
\end{theorem}

Although the proofs of both results are essentially the same, 
we separate them to avoid the cumbersome calculations associated with floor and ceiling functions. 
Thus, we will present only a brief proof of Theorem~\ref{main-theorem2}.

\medskip

The rest of this paper is organized as follows.
In Section 2, we prove Theorem~\ref{main-theorem}, and in Section 3, we prove Theorem~\ref{main-theorem2}.
In Section 4, we establish optimal results concerning paths of length one or two between equal-degree vertices, concluding with open problems that aim to generalize these results to paths of arbitrary length between equal-degree vertices.

\medskip

Throughout the paper, we use standard graph theory notation.
Let $G$ be a graph. For a vertex $x \in V(G)$, let $N_G(x)$ denote the set of vertices adjacent to $x$ in $G$, and let $\overline{N}_G(x) = V(G) \setminus (\{x\} \cup N_G(x))$ denote the set of vertices not adjacent to $x$ (excluding $x$ itself). Let $d_G(x) = |N_G(x)|$ be the \textit{degree} of $x$, and let $\Delta_G$ denote the maximum degree of vertices in $G$.
For disjoint vertex subsets $A, B \subseteq V(G)$, let $e_G(A)$ denote the number of edges with both endpoints in $A$, and let $e_G(A,B)$ denote the number of edges with one endpoint in $A$ and the other in $B$.
When the context is clear, we often drop the subscript $G$ from these notations.

\section{Proof of Theorem~\ref{main-theorem}}
Throughout this section, let $n \geq 600$, and let $G$ be a graph on $2n+1$ vertices with at least $n^2 + n$ edges,
such that $G$ contains no two vertices of the same degree joined by a path of length three.
Our goal is to prove that $G$ must be the complete bipartite graph $K_{n,n+1}$.
Suppose, for contradiction, that this is not the case.\footnote{Prior to presenting the proof, we remind the reader that nearly all arguments and calculations in the proof of Theorem~\ref{main-theorem} - except for the inequality \eqref{equ:CS2} in the proof of Lemma~\ref{lem:lower-Delta} - can be directly extended to Theorem~\ref{main-theorem2}, which concerns graphs with a sufficiently large even number of vertices. This extension requires only minor adjustments, and we would like to encourage the reader to keep this in mind throughout the proof.}

Recall that $\Delta$ denotes the maximum degree of $G$.
Our general proof strategy is to estimate $\Delta$ from both above and below (see Lemmas~\ref{lem:upper-Delta} and \ref{lem:lower-Delta}), leading to a contradiction.
Before proceeding, we note a crucial fact for our proof (though it is not needed until Lemma~\ref{distinct}):
by Mantel's theorem \cite{M}, we deduce that $G$ contains at least one triangle.

\medskip

\noindent{\bf \underline{Estimating $\Delta$ from above.}}
In the coming sequence of three lemmas (Lemmas~\ref{lem:degree-Nv}, \ref{lem:beta} and \ref{lem:upper-Delta}),
we aim to show that $\Delta\leq n+O(\sqrt{n})$.
In this process, we also need to study another graph parameter defined as follows: let $\beta \geq 0$ be the largest integer such that $G$ contains two vertices of degree $\beta$.
Since $G$ must have at least two vertices of the same degree, $\beta$ is well-defined.
As a warmup for the upcoming proof, we first show that
\begin{equation}\label{equ:beta}
2\leq \beta\leq \Delta.
\end{equation}
Suppose not. Then $\beta\leq 1$ and thus the total number $e(G)$ of edges in $G$ satisfies
$$e(G)=\frac{1}{2}\sum_{v\in V(G)}d(v)\le \frac{1}{2}\l 1+\sum_{k=1}^{2n}k \r= \frac{1}{2} \l 2n^2+n+1 \r<n^2+n,$$
which is contradictory to the assumption that $e(G)\geq n^2+n$. This proves \eqref{equ:beta}.

Our first lemma provides a useful tool for distinguishing vertices of different degrees.

\begin{lemma}\label{lem:degree-Nv}
Let $v \in V(G)$. If a vertex $u \in N(v)$ has at least two neighbors in $N(v)$,
then the degree $d(u)$ of $u$ in $G$ is distinct from the degrees $d(w)$ of all other vertices $w \in N(v) \setminus \{u\}$.
\end{lemma}

\begin{proof}
Suppose that $u_1$ and $u_2$ are two neighbors of $u$ in $N(v)$.
For any $w\in N(v)\backslash \{ u,u_1,u_2 \}$, we see that $u u_1 v w$ is a path of length three joining $u$ and $w$.
By our assumption, we conclude that $d(u)\neq d(w)$.
Also, since $u u_1 v u_2$ is a path of length three joining $u$ and $u_2$, we deduce that $d(u)\neq d(u_2)$;
similarly, we have $d(u)\neq d(u_1)$. This proves the lemma.
\end{proof}

The next lemma plays a key role in our proof and represents the most technical part of the argument.
Observe that $\Delta$ is at least the average degree of $G$, implying $\Delta \geq n + 1$.
This lemma suggests that either $\beta$ almost equals $\Delta$, or $G$ behaves really like a regular graph.

\begin{lemma}\label{lem:beta}
We have either $\beta \geq \Delta-1$ or $\Delta\leq n+1$.
\end{lemma}

\begin{proof}
Suppose for a contradiction that $\beta \leq \Delta-2$ and $\Delta\geq n+2$.
Then there exists a unique vertex $v_0$ in $G$ with $d(v_0)=\Delta$.
We divide $N(v_0)$ into the following two parts:
\begin{align}\notag
        A&:=\{ v\in N(v_0) ~|~ d(v)\geq 2n-\Delta+3 \},
        \\ \notag B&:=\{ v\in N(v_0)~ |~ d(v)\leq 2n-\Delta +2 \}.
\end{align}
Note that any $v\in A$ satisfies that $d(v)\in [2n-\Delta+3,\Delta-1]$.\footnote{This interval is well-defined as $\Delta\geq n+2$.}
For any vertex $v\in A$,
we see $|N(v)\cap N(v_0)|=|N(v)\setminus (\{v_0\}\cup \overline{N}(v_0))|\geq (2n-\Delta+3)-(2n+1-\Delta)=2$,
thus $v$ has at least two neighbors in $N(v_0)$.
By Lemma~\ref{lem:degree-Nv}, the degree of $v$ is distinct from the degree of any other vertex in $N(v_0)$.
In particular, all vertices in $A$ have distinct degrees.
Since there are at most $2\Delta - 2n - 3$ possible values for the degree of a vertex in $A$, we obtain
\begin{align}\label{equ:|B|>=}
2\Delta - 2n - 3 \geq |A| = \Delta - |B|, \mbox{ ~ or equivently ~ }  |B|\geq 2n-\Delta +3.
\end{align}

In the remainder of the proof, we estimate the number of edges in $G$ to derive a final contradiction.
As $\sum_{v\in B} d(v)$ is evidently bounded from above by definition,
our crucial step is to maximize $\sum_{v \in A \cup \overline{N}(v_0)} d(v)$, subject to the following:
\begin{itemize}
    \item[(a)] all vertices in $A$ have distinct degrees;
    \item[(b)] there are no two vertices in $A \cup \overline{N}(v_0)$ with the same degree greater than $\beta$.
\end{itemize}
We transform this to the following maximization problem:
Given integers $n,\Delta,\beta$ and $|B|$,
$$\mbox{ determine } \lambda(n,\Delta,\beta, |B|):=\max_{\mathcal{A}, \mathcal{B}} \sum_{k \in \mathcal{A} \cup \mathcal{B}} k, \mbox{ where } $$
\begin{itemize}
    \item[(1)] $\mathcal{A}$ is a sequence of $\Delta-|B|$ distinct integers in $\{1, 2, \dots, \Delta - 1\}$,\footnote{Let $\mathcal{A}$ be the sequence of degrees $d(v)$ for all $v \in A$.}
    \item[(2)] $\mathcal{B}$ is a sequence of $2n - \Delta$ integers in $\{0,1,\dots, \Delta - 1\}$,\footnote{Let $\mathcal{B}$ be the sequence of degrees $d(v)$ for all $v \in \overline{N}(v_0)$.} and
    \item[(3)] no two elements in $\mathcal{A} \cup \mathcal{B}$ share the same value greater than $\beta$.
\end{itemize}
Note that when $\lambda(n, \Delta, \beta, |B|)$ is maximized, every element in $\mathcal{B}$ is at least $\beta$, and we have
\begin{align}\label{equ:opt}
\sum_{v \in A \cup \overline{N}(v_0)} d(v)\leq \lambda(n,\Delta,\beta, |B|).
\end{align}

We divide the rest of the proof into three cases based on scenarios yielding different expressions for the function $\lambda(n, \Delta, \beta, |B|)$.

\bigskip

{\bf Case 1}: $|B|\leq \beta$.

\medskip

We have $|A| = \Delta - |B| \geq \Delta - \beta$. In this case, we claim that
\begin{align}\label{equ:case1}
\lambda(n, \Delta, \beta, |B|) = \sum_{k = |B|}^{\Delta - 1} k + (2n - \Delta)\beta.
\end{align}
To see this, we argue that when the maximum is achieved, $\mathcal{A} \supseteq \{\beta, \beta + 1, \ldots, \Delta - 1\}$.
Otherwise, there exist $k \in \mathcal{A}$ and $\ell \notin \mathcal{A}$ with $k < \beta \leq \ell \leq \Delta - 1$.
If $\ell$ is not included in $\mathcal{B}$, we replace $k$ with $\ell$ in $\mathcal{A}$.
If $\ell$ is included in $\mathcal{B}$, we replace $k$ with $\ell$ in $\mathcal{A}$ and also replace $\ell$ with $\beta$ in $\mathcal{B}$.
In both cases, conditions (1), (2), and (3) remain satisfied,
however the sum $\sum_{k \in \mathcal{A} \cup \mathcal{B}} k$ strictly increases, leading to a contradiction.
This shows that when the maximum is achieved, $\mathcal{A} \supseteq \{\beta, \beta + 1, \ldots, \Delta - 1\}$.
From this, we can easily deduce that all elements in $\mathcal{B}$ must be $\beta$, and then $\mathcal{A} = \{|B|, |B| + 1, \ldots, \Delta - 1\}$,
proving \eqref{equ:case1}.

Now using \eqref{equ:opt} and \eqref{equ:case1}, we can derive that
\begin{align}\notag
\sum_{v\in V(G)}d(v)&= \Delta+ \sum_{v\in A\cup \overline{N}(v_0)}d(v) + \sum_{v\in B}d(v)\le \Delta+\left(\sum_{k=|B|}^{\Delta-1}k + (2n-\Delta)\beta\right) + |B|(2n-\Delta+2)
\\ &= -\frac{1}{2}|B|^2+\l 2n-\Delta+\frac{5}{2}\r |B|+(2n-\Delta)\beta+\frac{1}{2}\Delta^2+\frac{1}{2}\Delta.\label{equ:Case1-1}
\end{align}
Here, we treat the above rightmost expression in \eqref{equ:Case1-1} as a quadratic polynomial $f_1(|B|)$ in the variable $|B|$.
By \eqref{equ:|B|>=}, we have $|B|\ge 2n-\Delta +3$, and $f_1(|B|)$ is decreasing in this range.
Thus, we deduce that
\begin{align}\notag
\sum_{v\in V(G)}d(v)\le f_1(|B|) &\le f_1(2n-\Delta +3)=2n^2+5n-2n\Delta+\Delta^2-2\Delta+3+(2n-\Delta)\beta.
\end{align}
Since $\sum_{v\in V(G)}d(v)=2e(G)\geq 2(n^2+n)$, we further derive
\begin{align}\label{18}
(2n-\Delta)\beta-2n\Delta+3n+\Delta^2 -2\Delta+3\ge 0.
\end{align}
Clearly $\Delta\leq 2n$.
If $2n-\Delta=0$, then inequality \eqref{18} is equivalent to $n\le 3$, a contradiction to our assumption that $n\geq 600$.
It follows that $2n-\Delta\geq 1$. Then \eqref{18} implies that
    \begin{align*}
        \beta \ge \frac{2n\Delta-3n-\Delta^2 +2\Delta-3}{2n-\Delta}= \Delta-2+\frac{n-3}{2n-\Delta}> \Delta-2,
    \end{align*}
which is contradictory to our assumption that $\beta \le \Delta-2$.
This finishes the proof for Case 1.

\bigskip

{\bf Case 2}: $|B|\geq \beta+1$ and $\Delta+|B|-2n\ge \beta+1$.

\medskip

Note that the combined sequence $\mathcal{A} \cup \mathcal{B}$ has $2n - |B|$ elements.
Since $\Delta + |B| - 2n \geq \beta + 1$,
condition (3) implies that the optimal configuration for $\mathcal{A} \cup \mathcal{B}$ is $\{\Delta - 1, \Delta - 2, \ldots, \Delta + |B| - 2n\}$,
for which conditions (1) and (2) remain satisfied. Thus, in this case, we have
\begin{align}\label{equ:case2}
\lambda(n, \Delta, \beta, |B|) = \sum_{k = \Delta + |B| - 2n}^{\Delta - 1} k.
\end{align}
Similarly to the previous case, using \eqref{equ:opt} and \eqref{equ:case2}, we can derive that
\begin{align}\notag
\sum_{v\in V(G)}d(v)&= \Delta+ \sum_{v\in A\cup \overline{N}(v_0)}d(v) + \sum_{v\in B}d(v)\le \Delta+ \left(\sum_{k=\Delta+ |B|-2n}^{\Delta-1}k\right)+ |B|(2n-\Delta+2)
\\ &=-\frac{1}{2}|B|^2+\l 4n-2\Delta+\frac{5}{2} \r |B| -2n^2+2n\Delta-n+\Delta.\label{equ:Case2-1}
\end{align}
Now we treat the above rightmost expression in \eqref{equ:Case2-1} as a quadratic polynomial $f_2(|B|)$ in the variable $|B|$.
Since $|B|$ is an integer, it is easy to deduce that
\begin{align}\label{equ:Case2-2}
\sum_{v\in V(G)}d(v)\le f_2(|B|)&\le f_2(4n- 2\Delta+2)=2\Delta^2-(6n+4)\Delta+6n^2+9n+3.
\end{align}
Since $\sum_{v\in V(G)}d(v)=2e(G)\ge 2(n^2+n)$, we have
\begin{align}\label{equ:Case2-3}
2\Delta^2-(6n+4)\Delta+4n^2+7n+3\ge 0.
\end{align}
Consider the expression above as a quadratic polynomial in $\Delta$, and elementary calculations show that either
\begin{align}\label{equ:case2-Delta}
\Delta\ge \frac{3n+2+\sqrt{n^2-2n-2}}{2}>2n, \mbox{ ~or~ } \Delta\le \frac{3n+2-\sqrt{n^2-2n-2}}{2}< n+2.
\end{align}
The first inequality cannot hold, and the second one contradicts our assumption that $\Delta \geq n + 2$,
thus completing the proof for Case~2.

\bigskip

{\bf Case 3}: $|B|\geq \beta+1$ and $\Delta+|B|-2n\leq \beta$.

\medskip

With the optimization reasoning from the previous cases in mind,
it is straightforward to see that the optimal configuration in this case consists of all elements of $\mathcal{A}$ and some elements of $\mathcal{B}$ forming $\{\Delta - 1, \Delta - 2, \ldots, \beta + 1\}$,
while the remaining elements of $\mathcal{B}$ are all $\beta$.
Therefore, in this case, we have
\begin{align}\label{equ:case3}
\lambda(n, \Delta, \beta, |B|)=\sum_{k=\beta+1}^{\Delta-1}k + (2n+1 -|B| -\Delta+\beta)\beta.
\end{align}
Then using \eqref{equ:opt} and \eqref{equ:case3}, we have that
\begin{align}\notag
\sum_{v\in V(G)}d(v)&= \Delta+ \sum_{v\in A\cup \overline{N}(v_0)}d(v) + \sum_{v\in B}d(v)\\ \notag
&\leq \Delta+\left(\sum_{k=\beta+1}^{\Delta-1}k + (2n+1 -|B| -\Delta+\beta)\beta\right) + |B|(2n-\Delta+2)\\ \label{equ:Case3-1}
&= \Delta+\sum_{k=\beta+1}^{\Delta-1}k + (2n+1-\Delta+\beta)\beta + |B|(2n-\Delta+2-\beta).
\end{align}
Now according to this formula, we further divide Case~3 into two subcases.

\bigskip

{\bf Subcase 1}: $2n-\Delta+2-\beta\geq 0$.
Since $|B|\leq \Delta$, we deduce from \eqref{equ:Case3-1} that
\begin{align}\notag
 \sum_{v\in V(G)}d(v) &\le \Delta+\sum_{k=\beta+1}^{\Delta-1}k + (2n+1 -\Delta+\beta)\beta + \Delta(2n-\Delta+2-\beta)
 \\ &=\frac{1}{2}\beta^2+\l 2n-2\Delta+\frac{1}{2}\r \beta+2n\Delta-\frac{1}{2}\Delta^2+\frac{5}{2}\Delta.\label{equ:Case3-sub1}
\end{align}
Consider the above rightmost expression in \eqref{equ:Case3-sub1} as a quadratic polynomial $f_3(\beta)$ in the variable $\beta$.
Using \eqref{equ:beta} and the assumption of this subcase, $2\leq \beta \leq 2n-\Delta+2$. Therefore,
\begin{align}\label{equ:max-1}
\sum_{v\in V(G)}d(v)\le f_3(\beta) \le \max \{ f_3(2), f_3(2n-\Delta+2) \}.
\end{align}

First, suppose the above maximum is achieved by $f_3(2)$. Then
\begin{align*}
\sum_{v\in V(G)}d(v)\le f_3(2)=-\frac{1}{2}\Delta^2 +\l 2n-\frac{3}{2}\r \Delta+4n+3.
\end{align*}
Since $\sum_{v\in V(G)}d(v)=2e(G)\geq 2(n^2+n)$, we derive
\begin{align*}
        -\frac{1}{2}\Delta^2 +\l 2n-\frac{3}{2}\r \Delta-2n^2+2n+3\ge 0.
\end{align*}
View the expression above as a quadratic polynomial in the variable \(\Delta\).
Then its discriminant
$\left(2n - \frac{3}{2}\right)^2 - 2(2n^2 - 2n - 3) = -2n + \frac{33}{4}$
must be non-negative, that is,
$n\leq \frac{33}{8}$, a contradiction.
Now we may assume that the maximum in \eqref{equ:max-1} is achieved by $f_3(2n-\Delta+2)$, implying
\begin{align*}
\sum_{v\in V(G)}d(v)\le f_3(2n-\Delta+2)= 2\Delta^2 -\l 6n+4\r \Delta+6n^2+9n+3.
\end{align*}
Note that this inequality is identical to \eqref{equ:Case2-2}.
We complete the proof of this subcase through the same analysis via \eqref{equ:Case2-3} and \eqref{equ:case2-Delta}.

\bigskip

{\bf Subcase 2}: $2n-\Delta+2-\beta< 0$. In this subcase,
the last term of the rightmost expression in \eqref{equ:Case3-1} increases as $|B|$ decreases.
Using the condition $|B|\ge \beta +1$, we deduce from \eqref{equ:Case3-1} that
\begin{align}\notag
\sum_{v\in V(G)}d(v) &\le \Delta+\sum_{k=\beta+1}^{\Delta-1}k + (2n -\Delta)\beta + (\beta +1)(2n-\Delta+2)
\\ &=-\frac{1}{2}\beta^2+\l 4n-2\Delta+\frac{3}{2} \r\beta +2n+\frac{1}{2}\Delta^2- \frac{1}{2}\Delta+2.\label{equ:Case3-sub2}
\end{align}
View the expression above as a quadratic  polynomial $f_4(\beta)$ in the variable $\beta$.
Since $\beta$ is an integer, \eqref{equ:Case3-sub2} implies that
\begin{align}\notag
\sum_{v\in V(G)}d(v) \le f_4(\beta) \le f_4(4n-2\Delta+1)=\frac{5}{2} \Delta^2 -\l 8n+\frac{7}{2} \r\Delta +8n^2+8n+3.
\end{align}
Since $\sum_{v\in V(G)}d(v)\geq 2(n^2+n)$, through elementary computation, this leads to that either
\begin{align}\notag
\Delta\ge \frac{8n+\frac{7}{2}+\sqrt{4n^2-4n-\frac{71}{4}}}{5} >2n \mbox{~ or ~}
\Delta\le \frac{8n+\frac{7}{2}-\sqrt{4n^2-4n-\frac{71}{4}}}{5}< \frac{6}{5}n+\frac{11}{10}.
\end{align}
The former inequality clearly cannot hold, so the later inequality $\Delta<\frac{6}{5}n+\frac{11}{10}$ holds.
This, together with the assumption that \(\beta \leq \Delta - 2\), implies that
\begin{align*}
\beta\le \Delta-2< 4n-2\Delta +\frac{3}{2}.
\end{align*}
Now using \eqref{equ:Case3-sub2} again, we deduce that
\begin{align*}
\sum_{v\in V(G)}d(v) \leq f_4(\beta) \leq f_4(\Delta-2) =-2\Delta^2+(4n+7)\Delta-6n-3.
\end{align*}
Substituting $\sum_{v\in V(G)}d(v)\ge 2(n^2+n)$, we derive
\begin{align}\label{equ:Case3-sub2-disc}
-2\Delta^2+(4n+7)\Delta-2n^2-8n-3\ge 0.
\end{align}
The discriminant of the leftmost expression above equals $(4n+7)^2-8(2n^2+8n+3)=25-8n$, which is strictly negative.
This is a contradiction to \eqref{equ:Case3-sub2-disc} for its non-negativity, proving Subcase~2.
Finally, we have completed the proof of Lemma~\ref{lem:beta}.
\end{proof}

Equipped with Lemma~\ref{lem:beta}, we are able to derive a satisfactory upper bound on $\Delta$ as follows.

\begin{lemma}\label{lem:upper-Delta}
It holds that $\Delta< n+\sqrt{2n}+\frac{3}{2}$.
\end{lemma}
\begin{proof}
By Lemma \ref{lem:beta}, if $\beta\leq \Delta-2$, then $\Delta\leq n-1< n+\sqrt{2n}+\frac{3}{2}$.
So we may assume that
\begin{align}\label{equ:b<=D}
\beta\geq \Delta-1.
\end{align}
Let $u$ and $v$ be two vertices in $G$ such that $d(u) = d(v)=\beta$.
Set $B:=N(u)\cap N(v)$, $A_u:=N(u)\backslash (\{v \} \cup B)$, $A_v:=N(v)\backslash (\{u \} \cup B)$ and $D:=\overline{N}(u)\cap \overline{N}(v)$.
Further, we set $x:=|A_u|$, which implies that $|A_v|=|A_u|=x$.
By the assumption, there does not exist a path of length three connecting $u$ and $v$.
    Thus, we have
    \begin{align}\label{empty}
        e(A_u,B)=e(A_v,B)=e(B)=e(A_u,A_v)=0.
    \end{align}

We distinguish the proof based on whether $uv\in E(G)$ or not.
If $u$ and $v$ are not adjacent, then $|B|=\beta-x\geq 0$ and $|D|=(2n-1)-\beta-x\geq 0$.
The total number of edges satisfies
\begin{align}\notag
        n^2+n&\leq e(G)=2\beta +e(A_u)+e(A_v)+e(A_u\cup B\cup A_v, D) +e(D)
        \\ \notag &\le 2\beta +\binom{x}{2}+\binom{x}{2} +\sum_{w\in D}d(w)\le 2\beta +(x^2-x) +(2n-1-\beta-x)\Delta
        \\ \notag &\le 2\beta +(x^2-x) +(2n-1-\beta-x)(\beta +1)= x(x-\beta -2) +2n\beta+2n -\beta^2-1
        \\ &\le 2n\beta+2n -\beta^2-1,\label{equ:upp-Del-1}
\end{align}
where the second last inequality holds by \eqref{equ:b<=D} and the last inequality follows from the facts that $x \geq 0$ and $x - \beta \leq 0$.
Solving this inequality \eqref{equ:upp-Del-1} leads to $\beta \le n + \sqrt{n - 1}$.
By \eqref{equ:b<=D}, we have $\Delta \le \beta + 1 \leq n + \sqrt{n - 1} + 1 < n + \sqrt{2n} + \frac{3}{2}$, as desired.

It remains to consider when $u$ and $v$ are adjacent.
Then $|B|=\beta-1-x\geq 0$ and $|D|=2n-\beta-x\geq 0$.
In this case, similarly, we can deduce that
    \begin{align}\notag
        n^2+n&=e(G)=(2\beta -1) +e(A_u)+e(A_v)+e(A_u\cup B\cup A_v, D) +e(D)
        \\ \notag &\le (2\beta -1) +\binom{x}{2}+\binom{x}{2} +\sum_{w\in D}d(w)\le (2\beta -1) +x^2-x +(2n-\beta-x)\Delta
        \\ \notag &\le (2\beta -1) +x^2-x +(2n-\beta-x)(\beta +1)= x(x-\beta -2) +2n\beta+2n+\beta -\beta^2-1
        \\ \notag &\le 2n\beta+2n+\beta -\beta^2-1.
    \end{align}
Solving this inequality gives that $\beta \le \frac{2n + 1 + \sqrt{8n - 3}}{2}$.
This, together with \eqref{equ:b<=D}, yields that $\Delta \leq \beta + 1 < n + \sqrt{2n} + \frac{3}{2}$,
completing the proof of this lemma.
\end{proof}

\medskip

\noindent{\bf \underline{Estimating $\Delta$ from below.}}
We turn to the estimation of bounding $\Delta$ from below.
In the remaining two lemmas, we shall show $\Delta > \frac{17}{16}n$.
It turns out that this can be deduced from the fact that there are $\Omega(n)$ vertices with pairwise distinct degrees, as we show in the next lemma.
Let $b_G$ denote the maximum number of triangles sharing a common edge in $G$.

\begin{lemma}\label{distinct}
There are at least $\frac{n}{2}+1$ vertices with pairwise distinct degrees in $G$.
\end{lemma}

\begin{proof}
Let $t_G$ denote the number of triangles in $G$.
As we pointed out at the beginning of the proof of Theorem~\ref{main-theorem},
we have $t_G\geq 1$ from Mantel's theorem \cite{M}.

For $0\le i\le 3$, let $t_i$ be the number of pairs $(v, T)$,
where $T$ is a triangle in $G$ and $v\notin V(T)$ is a vertex in $G$ such that $v$ is adjacent to exactly $i$ vertices in $T$.
It directly follows that
\begin{align}\label{A}
t_0+t_1+t_2+t_3= t_G\cdot (2n-2).
\end{align}
For distinct vertices $u$ and $v$, let $\widehat{d}(u,v) = |N(u) \cap N(v)|$ be the number of common neighbors of $u$ and $v$.
Let us consider the sum
\begin{align}\label{M1}
M:=\sum_{uv\in E(G)} \widehat{d}(u,v)\cdot |\overline{N}(u)\cap \overline{N}(v)|.
\end{align}
Every pair $(v, T)$ with $v$ adjacent to exactly one vertex in $T$ is counted once in $M$,
while every pair $(v, T)$ with $v$ adjacent to no vertex in $T$ is counted three times in $M$.
Thus,
\begin{align}\label{M4}
M=t_1+3t_0.
\end{align}
Since $\widehat{d}(u,v)\leq b_G$ for all edges $uv\in E(G)$, it follows by \eqref{M1} that
\begin{align}\label{M2}
M\le b_G \sum_{uv\in E(G)} |\overline{N}(u)\cap \overline{N}(v)|.
\end{align}
For any two adjacent vertices $u$ and $v$, we have
\begin{align}\notag
|\overline{N}(u)\cap \overline{N}(v)|=2n+1+ \widehat{d}(u,v)-d(u)-d(v).
\end{align}
Also observing that
\begin{align*}
\sum_{uv\in E(G)} \widehat{d}(u,v)= 3 t_G \mbox{ ~and~} \sum_{uv\in E(G)}\left(d(u)+d(v)\right)= \sum_{v\in V(G)} d^2(v),
\end{align*}
we see that \eqref{M2} can be transformed into
\begin{align}\notag
M&\le b_G \sum_{uv\in E(G)} \left(2n+1+ \widehat{d}(u,v) -d(u)-d(v)\right)\\
& = b_G \l(2n+1)e(G)+ 3 t_G -\sum_{v\in V(G)} d^2(v) \r.\label{M3}
\end{align}
If $e(G)\ge n^2+n+1$, it follows by the Cauchy-Schwarz inequality that
\begin{align}\label{equ:CS1}
\sum_{v\in V(G)} d^2(v)\ge \frac{\l \sum_{v\in V(G)} d(v) \r^2}{2n+1}= \frac{\l 2e(G) \r^2}{2n+1}\geq (2n+1)e(G);
\end{align}
otherwise, we have $\sum_{v\in V(G)}d(v)=2e(G)= 2(n^2+n)$.
Then, using integrality and convexity,
\begin{align}\label{equ:CS2}
\sum_{v\in V(G)} d^2(v)\geq n\cdot (n+1)^2+(n+1)\cdot n^2=(2n+1)e(G).
\end{align}
Therefore, in either case we can derive from \eqref{M3} that
\begin{align}\label{M5}
M&\le 3b_G t_G.
\end{align}
Now substituting \eqref{M4} and \eqref{M5} into \eqref{A}, we have
\begin{align*}
t_2+t_3 = t_G\cdot (2n-2)-t_0-t_1\geq t_G\cdot(2n-2)-M\geq t_G\cdot (2n-2-3b_G).
\end{align*}

Recall the definition of $t_2$ and $t_3$.
By averaging, there must be a triangle say $T_0$ in $G$ such that there are at least $2n - 2 - 3b_G$ vertices in $V(G) \setminus V(T_0)$ each of which is adjacent to at least two vertices in $T_0$.
Note that
\begin{align}\notag
\max\{ b_G+1,2n-3b_G\}\geq \frac{1}{4}\left( 3(b_G+1)+(2n-3b_G)\right)=\frac{n}{2}+\frac{3}{4}.
\end{align}
Since $\frac{n}{2}+\frac{3}{4}$ cannot be an integer, we deduce that
\begin{align}\label{M6}
\max\{ b_G+1,2n-3b_G\}\geq \frac{n}{2}+1.
\end{align}

Now we consider two cases according to \eqref{M6}.
First, assume $b_G + 1 \ge \frac{n}{2} + 1$.
Let $u$ and $v$ be two adjacent vertices with $b_G$ common neighbors.
For any $w_1, w_2\in N(u) \cap N(v)$,
$w_1 u v w_2$ and $u w_1 v w_2$ are both paths of length three.
Thus, we conclude that all $b_G + 1$ vertices in $\{u\} \cup (N(u) \cap N(v))$ have pairwise distinct degrees.
Since $b_G + 1 \ge \frac{n}{2} + 1$, we prove this case.

It remains to consider the case when $2n - 3b_G\geq \frac{n}{2} + 1$.
Let $uvw$ be the triangle $T_0$ such that there are at least $2n - 2 - 3b_G$ vertices in $V(G) \setminus \{u, v, w\}$, each adjacent to at least two vertices in $\{u, v, w\}$.
Let $x_1, x_2$ be any two such vertices in $V(G) \setminus \{u, v, w\}$, both adjacent to at least two vertices in $\{u, v, w\}$.
Without loss of generality, assume $x_1u, x_1v, x_2v \in E(G)$.
Then the paths $x_1 u v x_2$, $u w v x_1$, $v w u x_1$, and $w u v x_1$ are all of length three, implying that $d(x_1)$ is distinct from $d(x_2)$, $d(u)$, $d(v)$, and $d(w)$.
Additionally, $u x_1 v w$ is a path of length three.
Thus, all $2n - 2 - 3b_G$ vertices, together with $u$ and $w$, have pairwise distinct degrees.
Since $2n - 3b_G \geq \frac{n}{2} + 1$, the proof of the lemma is complete.
\end{proof}

Using Lemma~\ref{distinct}, we can promptly derive the following lower bound on $\Delta$.

\begin{lemma}\label{lem:lower-Delta}
It holds that $\Delta> \frac{17}{16}n$.
\end{lemma}

\begin{proof}
We estimate the sum of the degrees of all vertices in $G$.
Since Lemma \ref{distinct} indicates that there are at least $\frac{n}{2}+1$ vertices with pairwise distinct degrees in $G$,
we derive that
\begin{align}\notag
2(n^2+n)\leq 2e(G)=\sum_{v\in V(G)}d(v) \le \sum_{k=0}^{\lceil n/2 \rceil} (\Delta -k) + \l 2n-\left\lceil \frac{n}{2} \right\rceil \r \Delta\le (2n+1)\Delta-\frac{n^2+2n}{8}.
\end{align}
This implies that
\begin{align*}
\Delta\ge \frac{17n^2+18n}{8(2n+1)}> \frac{17}{16}n,
\end{align*}
completing the proof of Lemma~\ref{lem:lower-Delta}.
\end{proof}

\medskip

Using Lemmas~\ref{lem:upper-Delta} and \ref{lem:lower-Delta},
we can obtain that
\[
\frac{17}{16} n < \Delta < n + \sqrt{2n} + \frac{3}{2}.
\]
Through some elementary calculations, this shows that $n \leq 558$, which contradicts the assumption that $n \geq 600$.
This final contradiction completes the proof of Theorem~\ref{main-theorem}.
\qed

\bigskip

We note that the constant $600$ in Theorem~\ref{main-theorem} can be further improved with more refined calculations.
In particular, a more careful estimation of the lower bound on $\Delta$ in Lemma~\ref{lem:lower-Delta},
using a similar approach as in Lemma~\ref{lem:upper-Delta}, would reduce this bound to below $150$.
However, for the sake of brevity and clarity, we present the proof in its current simplified form.

\section{Proof of Theorem \ref{main-theorem2}}
In this section, let $n \geq n_0$ be sufficiently large, and let $G$ be a graph with $2n$ vertices and at least $n^2 - 1$ edges.
Suppose that $G$ does not contain two vertices of the same degree that are joined by a path of length three.
We aim to show that $G$ must be the complete bipartite graph $K_{n-1,n+1}$.
We point out that most of the proof follows the same arguments as the proof of Theorem~\ref{main-theorem}, requiring only minor modifications.
However, certain parts of the proof require additional treatment, and we provide the necessary details here.

Let $\Delta$ be the maximum degree of $G$.
Let $\beta$ be the largest integer such that there are two vertices of the same degree $\beta$ in $G$.
Lemma~\ref{lem:degree-Nv} remains hold evidently.
Through similar calculations, we can prove that the same conclusion of Lemma~\ref{lem:beta} still holds.
From this, we show that the conclusion of Lemma~\ref{lem:upper-Delta} can be modified to $\Delta < n + \sqrt{2n} + 1$.

Next, we require two results on triangles in graphs as follows.
The first result is a strengthening of Mantel's theorem (see, e.g., \cite{B81}).
Let $K_{n,n}^-$ be the graph obtained from the complete bipartite graph $K_{n,n}$ by deleting an edge.

\begin{theorem}[A strengthening of Mantel's theorem; see \cite{B81}]\label{thm:triangle-1}
Let $G$ be a graph with $2n$ vertices and at least $n^2 - 1$ edges.
If $G$ has no triangles, then $G$ is either $K_{n-1,n+1}$, $K_{n,n}^-$, or $K_{n,n}$.
\end{theorem}

The second result is a supersaturation-type result, which is implicit in the literature but can be derived from the arguments of many papers, e.g., \cite{E,Mubayi}.
It is worth pointing out that for our purposes, any $\Omega(n)$ lower bound on the number of triangles suffices.

\begin{theorem}[\cite{E,Mubayi}]\label{thm:triangle-2}
Let $n$ be sufficiently large, and let $G$ be a graph with $2n$ vertices and at least $n^2 - 1$ edges.
If $G$ contains at least one triangle, then $G$ has at least $n - 2$ triangles.\footnote{This bound is tight, as demonstrated by the graph obtained from $K_{n-1,n+1}$ by adding an edge $xy$ in the part of size $n+1$ and deleting any edge incident to $x$, as well as the graph obtained from $K_{n,n}$ by adding an edge $xy$ in one part and deleting any two edges incident to $x$.}
\end{theorem}

If $G$ has no triangles, then by Theorem~\ref{thm:triangle-1}, $G$ is either $K_{n-1,n+1}$, $K_{n,n}^-$, or $K_{n,n}$,
where the first graph is what we aim to show, while the other two graphs both contain two vertices of the same degree which are joined by a path of length three.
Therefore, we may assume that $G$ contains at least one triangle.
By Theorem~\ref{thm:triangle-2}, the number $t_G$ of triangles in $G$ satisfies $t_G \geq n - 2$.
Analogous to Lemmas~\ref{distinct} and \ref{lem:lower-Delta}, we have the following result.

\begin{lemma}\label{lem:2n-many-degrees}
There are at least $\frac{2n+6}{7}$ vertices with pairwise distinct degrees in $G$, and $\Delta>\frac{50}{49}n$.
\end{lemma}

\begin{proof}
We define $M$ to be the same expression as in \eqref{M1}.
Following the arguments of Lemma~\ref{distinct}, we establish the following analogue of \eqref{M3}:
    \begin{align}\label{equ:even-M1}
        M\le b_G \l 2n\cdot e(G)+ 3 t_G -\sum_{v\in V(G)} d^2(v) \r.
    \end{align}
If $e(G)\geq n^2$, it follows by the Cauchy-Schwarz inequality that
$$\sum_{v\in V(G)} d^2(v)\ge \frac{\l \sum_{v\in V(G)} d(v) \r^2}{2n}= \frac{\l 2e(G) \r^2}{2n} \ge 2n\cdot e(G);$$
otherwise, we have $\sum_{v\in V(G)} d^2(v)=2e(G)= 2(n^2-1)$.
Then, using integrality and convexity, we obtain
\begin{align*}
\sum_{v\in V(G)} d^2(v)\ge (2n-2)\cdot n^2+2\cdot (n-1)^2= 2n\cdot e(G)-2n+2,
\end{align*}
where equality is achieved by the degree sequence consisting of $2n - 2$ vertices of degree $n$ and $2$ vertices of degree $n - 1$.
Using these estimations, it follows from \eqref{equ:even-M1} that
\begin{align*}
M\le b_G \l 2n-2+ 3 t_G \r,
\end{align*}
and consequently, we have
\begin{align}\notag
t_2+t_3 = t_G\cdot(2n-3)-t_0-t_1\geq t_G\cdot(2n-3)-M\geq t_G\cdot\l 2n-3-3b_G-\frac{(2n-2)b_G}{t_G}\r.
\end{align}
Since $t_G\ge n-2$, we further have
\begin{align}\notag
t_2+t_3\ge t_G\cdot\l 2n-3-3b_G-\frac{(2n-2)b_G}{n-2}\r > t_G\cdot \l 2n-3-6b_G\r.
\end{align}
Then there must be a triangle $T_0$ such that there are at least $2n-2-6b_G$ vertices in $V(G)\backslash V(T_0)$,
each of which is adjacent to at least two vertices in $T_0$.
Also note that
\begin{align}\notag
\max\{ b_G+1,2n-6b_G\}\ge \frac{1}{7}\big( 6(b_G+1)+(2n-6b_G)\big)=\frac{2n+6}{7}.
\end{align}
Following an argument analogous to Lemma~\ref{distinct},
we conclude that $G$ contains at least $\frac{2n+6}{7}$ vertices with pairwise distinct degrees.

Finally, adapting the proof of Lemma~\ref{lem:lower-Delta}, we obtain the lower bound $\Delta > \frac{50}{49}n$.
\end{proof}

Combining the above bounds on $\Delta$, we obtain
\begin{align}
    \frac{50}{49}n < \Delta < n + \sqrt{2n} + 1,
\end{align}
which implies $n < 5000$. However, this contradicts our assumption that $n \geq n_0$ is sufficiently large.
This final contradiction completes the proof of Theorem~\ref{main-theorem2}.
\qed

\section{Discussion and open problems}
In this paper, we resolve a problem of Erd\H{o}s and Hajnal \cite{E91} by proving that for $n \geq 600$,
every $(2n+1)$-vertex graph with at least $n^2 + n + 1$ edges contains two vertices of the same degree joined by a path of length three.
We also establish a tight analogous result for graphs with even order.

To conclude, we discuss extensions of our results and propose open problems for future investigation.
A natural direction is to study the extremal function defined below.

\begin{definition}
For any positive integers $\ell$ and $n$, let $p_\ell(n)$ denote the maximum number of edges in an $n$-vertex graph $G$ that contains no two vertices of equal degree connected by a path of length $\ell$.
\end{definition}

Using this definition, our main results (Theorems~\ref{main-theorem} and~\ref{main-theorem2}) yield the exact formulas:
$$
p_3(2n+1) = n^2 + n \quad \text{and} \quad p_3(2n) = n^2-1  \quad \mbox{for all } n \geq n_0.
$$

This section is organized as follows.
In Subsections~\ref{subsec:p1} and~\ref{subsec:p2}, we establish tight bounds for $p_1(n)$ and $p_2(n)$, respectively.
We conclude in Subsection~\ref{subsec:open} with open problems concerning $p_\ell(n)$ for general $\ell$.

\subsection{On adjacent vertices with equal-degree}\label{subsec:p1}
In this subsection, we establish the following bound for $p_1(n)$.
It yields $p_1(n) = \frac{n^2}{2} - \frac{n\sqrt{2n}}{3} + O(n)$, determining its exact order of magnitude.

\begin{theorem}\label{length1}
For any positive integer $n$, it holds that
\begin{align}\notag
        p_1(n)\le \frac{n(n-m-1)}{2}+ \frac{m(m+1)(m+2)}{12},\ {\rm where}\ m=\left\lfloor \frac{-1+\sqrt{8n+1}}{2} \right\rfloor,
\end{align}
where equality holds if $\frac{-1 + \sqrt{8n + 1}}{2}$ is an integer.
\end{theorem}

\begin{proof}
Let $G$ be an $n$-vertex graph which does not contain two adjacent vertices of the same degree.
Then the complement graph $\overline{G}$ of $G$ does not contain two non-adjacent vertices of the same degree.
For any $k\geq 1$, let $n_k$ denote the number of vertices in $\overline{G}$ of degree $k-1$.
Then these $n_k$ vertices are pairwise adjacent in $\overline{G}$.
Thus, we deduce that $n_k\le k$.
Set
$$m:=\left\lfloor \frac{-1+\sqrt{8n+1}}{2} \right\rfloor,$$
to be the unique positive integer satisfying $\sum_{k=1}^m k\le n< \sum_{k=1}^{m+1} k.$
Let $r:=n-\sum_{k=1}^m k.$
Now we consider the sum of degrees of all vertices in $\overline{G}$.
Since $\overline{G}$ contains at most $k$ vertices of degree $k-1$ for any integer $k\geq 1$,
we derive that
\begin{align*}
\sum_{v\in V(G)} d_{\overline{G}}(v)\ge \sum_{k=1}^m k(k-1) +rm=nm-\frac{m(m+1)(m+2)}{6}.
\end{align*}
It follows that
$$e(G)=\binom{n}{2}-\frac{1}{2}\sum_{v\in V(G)} d_{\overline{G}}(v) \le \frac{n(n-m-1)}{2}+ \frac{m(m+1)(m+2)}{12},$$
providing the desired upper bound on $p_1(n)$.
This bound is sharp when $m$ is a positive integer and $n=m(m+1)/2$.
Let $\overline{G}$ be the disjoint union of cliques of sizes $1, 2, \ldots, m$.
Then the corresponding graph $G$ attains the maximum number of edges specified by the bound.
\end{proof}

\subsection{On paths with equal-degree endpoints of length two}\label{subsec:p2}
In this subsection, we investigate the function $p_2(n)$.

We first present a lower bound construction. The \emph{half graph} $H_n$ on $2n$ vertices is a bipartite graph with vertex parts $\{u_1,\ldots,u_n\}$ and $\{v_1,\ldots,v_n\}$, where $u_i$ is adjacent to $v_j$ if and only if $i \leq j$.
Observe that $H_n$ contains no two vertices of the same degree connected by a path of length two. Therefore, we obtain the lower bound
$p_2(2n) \geq \frac{n(n+1)}{2}.$
The following result demonstrates that this bound is tight.

\begin{theorem}\label{length2}
For any positive integer $n$, it holds that
\begin{align}\notag
p_2 (2n)= \frac{n(n+1)}{2}.
\end{align}
\end{theorem}

\begin{proof}
Let $G$ be a graph with $2n$ vertices which does not contain two vertices of the same degree joined by a path of length two.
Let $\Delta$ be the maximum degree of $G$ and let $v$ be a vertex of degree $\Delta$.
Set $N(v):=\{ v_1,v_2,...,v_\Delta \}$.
By the assumption, the $\Delta$ vertices in $N(v)$ have pairwise distinct degrees.
Thus, we may assume that $d(v_i)=i$ for $1\le i\le \Delta$.
Further set $N(v_\Delta):=\{v,u_1,u_2,...,u_{\Delta-1} \}$.
Similarly, we may assume that $d(u_i)=i$ for $1\le i\le \Delta-1$.

Since $d(v)=d(v_\Delta)$, we must have $N(v)\cap N(v_\Delta)=\emptyset$, which implies that $\Delta\leq n$.
Therefore, we can derive
\begin{align}\notag
2e(G)&=\sum_{v\in V(G)}d(v)=\sum_{v\in N(v)} d(v) + \sum_{v\in N(v_\Delta)} d(v) + \sum_{v\in \overline{N}(v)\cap \overline{N}(v_\Delta)} d(v)
        \\ \notag &\le 2\sum_{k=1}^\Delta k+ (2n-2\Delta)\Delta=-\Delta^2+(2n+1)\Delta\le n^2+n,
\end{align}
where the last inequality holds as $\Delta\leq n$.
Thus, we have $p_2(2n)\leq n(n+1)/2$.
Combined with the matching lower bound, this completes the proof
\end{proof}

We note that the half graph $H_n$ is not the unique extremal graph.
When $n$ is even, let $G_n$ be obtained from $H_n$ by deleting edges $u_{n/2}v_{n/2}$ and $u_{n/2+1}v_{n/2+1}$, and adding edges $u_{n/2}u_{n/2+1}$ and $v_{n/2}v_{n/2+1}$.
One can verify that $G_n$ also contains no two vertices of the same degree connected by a path of length two.

\subsection{Problems on paths with equal-degree endpoints of any length}\label{subsec:open}
We conclude by proposing several open problems about $p_\ell(n)$ for general $\ell$,
which extend the original problem of Erd\H{o}s and Hajnal (Problem~\ref{main problem}).
For odd lengths, we conjecture that the complete bipartite graph $K_{n,n+1}$ remains extremal.

\begin{conjecture}\label{conjecture1}
For any odd integer $\ell\geq 3$ and sufficiently large $n$, it holds that
\[
p_{\ell}(2n+1) = n^2 + n.
\]
\end{conjecture}

We tend to believe that the behavior differs significantly between odd and even $\ell$.
For even $\ell \geq 2$,
the half graph $H_n$ contains no two vertices of equal degree connected by a path of length $\ell$,
but any additional edge creates such a pair.
This yields the lower bound
\[
p_\ell(2n) \geq \frac{n(n+1)}{2}.
\]
While $H_n$ provides this bound, it remains unclear whether it is optimal. We therefore pose:

\begin{problem}\label{problem}
Determine the exact value of $p_\ell(2n)$ for all even $\ell$ and sufficiently large $n$.
\end{problem}

\bigskip
\medskip

{\noindent\bf Acknowledgments.}
We thank Long-Tu Yuan for bringing reference \cite{B81} to our attention.

\end{document}